\documentclass{cedram-aif}
\usepackage{amsmath,latexsym,epsfig,graphics,amsfonts,amssymb,mathrsfs}

\equalenv{proposition}{prop}
\equalenv{corollary}{coro}

\renewcommand{\leq}{\leqslant}
\renewcommand{\geq}{\geqslant}

\newcommand{\folF}{\mathscr{F}}
\newcommand{\folR}{\mathscr{R}}

\title{Flowability of plane homeomorphisms}

\author[F. Le Roux]{\firstname{Fr\'ed\'eric} \lastname{Le Roux}}
\address{Universit\'e Paris Sud, Laboratoire de math\'ematiques, Bat. 425, 91405 Orsay Cedex, France}
\email{frederic.le-roux@math.u-psud.fr}

\author[A. G. O'Farrell]{\firstname{Anthony}  \middlename{G.} \lastname{O'Farrell}}
\address{Department of Mathematics, Logic House, National Univeristy of Ireland Maynooth, Maynooth, County Kildare, Ireland}
\email{anthonyg.ofarrell@gmail.com}

\author[M. Roginskaya]{\firstname{Maria} \lastname{Roginskaya}}
\address{Department of Mathematics, Chalmers University of Technology, S-412 96 G\H{o}teborg, Sweden}
\email{maria@math.chalmers.se}

\author[I. Short]{\firstname{Ian} \lastname{Short}}
\address{Department of Mathematics and Statistics, The Open University, Milton Keynes, MK7 6AA, United Kingdom}
\email{i.short@open.ac.uk}
\thanks{The second and fourth authors' work was supported by
Science Foundation Ireland grant 05/RFP/MAT0003. The second author was also supported by the ESF Network HCAA}



\begin{document}

\begin{abstract}
We describe necessary and sufficient conditions for an orientation preserving fixed point free planar homeomorphism that preserves the standard Reeb foliation to embed in a planar flow that leaves the foliation invariant.
\end{abstract}

\maketitle

%%%%%%%%%%%%%%%%%%%%%%%%%%%%%%%%%%%%%%%%%%%%%
\section{Introduction}\label{S: introduction}
%%%%%%%%%%%%%%%%%%%%%%%%%%%%%%%%%%%%%%%%%%%%%

The object of this paper is to give necessary and sufficient conditions for an orientation preserving fixed point free planar homeomorphism that preserves the standard Reeb foliation to embed in a planar flow that leaves the foliation invariant.

A \emph{nonsingular oriented foliation of the plane} is a collection $\folF$ of oriented open arcs, called \emph{leaves}, that partition $\mathbb{R}^2$ and satisfy the following property. For each point $p$ in $\mathbb{R}^2$ there is an open neighbourhood $N$ of $p$ and an orientation preserving homeomorphism $\varphi: N\rightarrow ]0,1[^2$, called a \emph{chart} for $\folF$, such that if $L$ is a leaf of $\folF$ that intersects $N$, then every connected component of
 $\varphi (N\cap L)$ is a horizontal segment $]0,1[\times\{a\}$, for some point $a$ in $]0,1[$. The orientation of $N\cap L$ must correspond under $\varphi$ to the usual orientation of $]0,1[\times\{a\}$. We denote the leaf of a point $x$ by $L(x)$,  and, given points $x_1$ and $x_2$ on $L(x)$, we use the notation $x_1<x_2$ to mean that $x_1$ occurs before $x_2$ according to the ordering of $L(x)$.

The simplest such foliation consists of a complete collection of oriented parallel lines. Foliations homeomorphic to this foliation are known as \emph{trivial foliations}. The simplest homeomorphism class of non-trivial foliations are the \emph{Reeb foliations}---one is shown in Figure~\ref{F: Reeb}---and it is to these foliations that we pay most attention. We can describe explicitly an oriented Reeb foliation $\folR$ as the collection of lines $\{(t,v)\,:\, t\in\mathbb{R}\}$, for $|v|\geq 1$, together with the collection of lines
\[
\left\{\left(u-\left(\cos\left(\frac{\pi}{2}t\right)\right)^{-1},t\right)\,:\, -1<t<1\right\},
\]
 for $u$ in $\mathbb{R}$. The line $\Delta$ given by the equation $y=-1$, and  the line $\Delta'$ given by the equation $y=1$, are the \emph{nonseparated leaves} of $\folR$.  We orient $\Delta$ to the right and $\Delta'$ to the left; this produces a unique orientation on $\folR$ which is indicated by arrows in Figure~\ref{F: Reeb}. Let
\begin{eqnarray*}
U_1 = \{(x,y)\,:\, y<-1\}, \quad
U_2 = \{(x,y)\,:\, -1<y<1\},\\
U_3 = \{(x,y)\,:\, y>1\}.
\end{eqnarray*}

\begin{figure}[ht]
\centering
\includegraphics[scale=1.0]{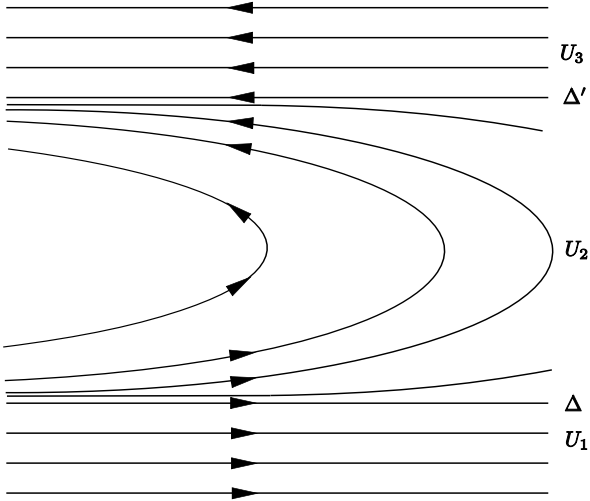}
\caption{The Reeb foliation.}
\label{F: Reeb}
\end{figure}
Refer to \cite{HaRe1957} for an excellent introduction to plane foliations.

A \emph{flow} on a topological space $X$ is a continuous homomorphism $t\mapsto \Phi^t$ from the group $(\mathbb{R},+)$ to the group of homeomorphisms of $X$, equipped with the compact-open topology. We denote this flow by $(\Phi^t)$. We shall only study flows on $\mathbb{R}$ and $\mathbb{R}^2$. A \emph{Brouwer homeomorphism} is an orientation preserving homeomorphism of $\mathbb{R}^2$ which has no fixed points. Given a nonsingular oriented foliation  $\folF$ of $\mathbb{R}^2$, an \emph{$\folF$-homeomorphism} is a Brouwer homeomorphism $f$ such that, for each leaf $L$ of $\folF$, $f(L)=L$ and $f(x)>x$ for each element $x$ of $L$. A flow $(\Phi^t)$ is an \emph{$\folF$-flow} if $\Phi^t$ is an $\folF$-homeomorphism for each positive number $t$.

We address the problem of characterizing those  $\folF$-homeomorphisms that arise as the time one map of an $\folF$-flow. In other words, given an $\folF$-homeomorphism $f$, we ask whether there is an $\folF$-flow $(\Phi^t)$ such that $f=\Phi^1$. Such maps $f$ we call \emph{$\folF$-flowable}. There are examples of $\folF$-homeomorphisms that do not arise as the time one map of any planar flow at all (see \cite{BeLe2003}). The problem of characterizing $\folF$-flowable maps is straightforward when $\folF$ is the trivial foliation: each  $\folF$-homeomorphism is $\folF$-flowable, and there is only one conjugacy class of flows. See  \cite{An1965,Jo1972,Jo1982,Ut1982} for related literature. We focus primarily on the Reeb foliation $\folR$, and we characterize the $\folR$-flowable $\folR$-homeomorphisms. In contrast to the trivial foliation, there are uncountably many conjugacy classes of $\folR$-flows (see \cite{Le1999}).

Our techniques extend easily to more general plane foliations; for example, those foliations described by the last theorem in \cite{Go1972}, each of which has finitely many nonseparated leaves. It is more difficult to characterize the $\folF$-flowable $\folF$-homeomorphisms when the foliation $\folF$ has a nonseparated leaf which is the limit of a sequence of nonseparated leaves.

We now describe two conditions which characterize the $\folR$-flowable $\folR$-homeomorphisms. Recall that $\Delta$ and $\Delta'$ are the nonseparated leaves of $\folR$.

An $\folR$-homeomorphism $f$ has the \emph{four point matching property} when, given
\begin{itemize}
\item a pair $(x,x')$ in $\Delta\times\Delta'$,
\item a sequence $(x_k)$ in $\mathbb{R}^2$ such that $x_k\to x$ and $f^{k}(x_k)\to x'$, and
\item a sequence $(y_k)$ in $\mathbb{R}^2$ such that $y_k\in L(x_k)$,
\end{itemize}
it follows that the sequence $(y_k)$ converges to a point $y$ in $\Delta$ if and only if the sequence
$(f^{k}(y_k))$ converges to a point $y'$ in $\Delta'$. If $f$ is the time one map of  an $\folF$-flow $(\Phi^t)$
then $f$ has the four point matching property. Indeed, if we assume that $(y_k)$ converges to $y$ then, given the three conditions stated above, there are unique real numbers $t$ and $(t_k)$  such that $\Phi^t(x)=y$  and $\Phi^{t_k}(x_k)=y_k$ for each positive integer $k$. Since $(y_k)$ converges to $y$, it follows that $(t_k)$ converges to $t$. Given that $f^{k}(y_k)= \Phi^{t_k}(f^{k}(x_k))$, we deduce that the sequence $(f^{k}(y_k))$ converges to $\Phi^t(x')$. The converse can be verified in a similar manner.

Assuming the four point matching property, we  define a relation $\sim_f$ on $\Delta\times\Delta'$ by the property that $(x,x')\sim_f(y,y')$ if and only if there are sequences $(x_k)$ and $(y_k)$ as stated above such that $x_k\to x$, $y_k\to y$, $f^{k}(x_k)\to x'$, and $f^{k}(y_k)\to y'$. In Section~\ref{S: DeltaDelta'} we prove that the relation $\sim_f$  is an equivalence relation. The meaning of this equivalence relation when $f$ is an $\folR$-flowable map is that the time taken to flow from $x$ to $y$ is equal to the time taken to flow from $x'$ to $y'$. Using $\sim_f$ we can define our second condition for characterizing $\folR$-flowable homeomorphisms.

An $\folR$-homeomorphism $f$ that satisfies the four point matching property is said to satisfy the \emph{eight point matching property} if,
whenever we are given points
$x_1$, $y_1$, $x_2$, and $y_2$ in $\Delta$ and
$x'_1$, $y'_1$, $x'_2$, and $y'_2$ in $\Delta'$, then
each three of the following four conditions
implies the remaining one:
\begin{itemize}
\item $(x_1,x'_1)\sim_f(y_1,y'_1)$,
\item $(x_1,x'_2)\sim_f(y_1,y'_2)$,
\item $(x_2,x'_1)\sim_f(y_2,y'_1)$,
\item $(x_2,x'_2)\sim_f(y_2,y'_2)$.
\end{itemize}
A  time one map $f$ of  a flow $(\Phi^t)$ has the eight point matching property because each equivalence asserts that the time to flow between a pair of points on $\Delta$ is equal to the time to flow between a pair of points on $\Delta'$. Figure~\ref{F: relation2} illustrates the  eight point matching property.

\begin{figure}[ht]
\centering
\includegraphics[scale=1.0]{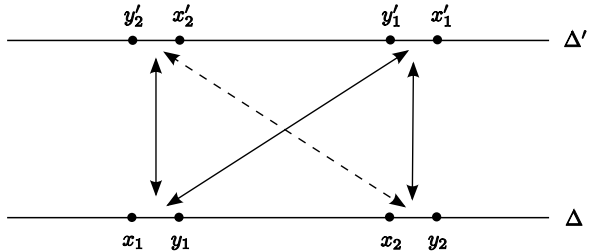}
\caption{The eight point matching property.}
\label{F: relation2}
\end{figure}

Given both the four and eight point matching properties there is a second equivalence relation $\approx_f$, which this time is defined on the set $\Delta\times\Delta$. It is given by the property that $(x_1,y_1)\approx_f(x_2,y_2)$ if and only if there is a pair $(x',y')$ in $\Delta'\times \Delta'$ such that both $(x_1,x')\sim_f (y_1,y')$ and $(x_2,x')\sim_f (y_2,y')$. We prove that $\approx_f$ is an equivalence relation in Section~\ref{S: DeltaDelta}.

We are now in a position to state our main theorem.

\begin{theorem}\label{T: main}
Let $\folR$ be the Reeb foliation and let $f$ be an $\folR$-homeomorphism. Then $f$ is $\folR$-flowable if and only if $f$ has both the four point and eight point matching properties.
Furthermore, if there are two $\folR$-flows $(\Phi^t)$ and $(\Psi^t)$
such that $\Phi^1=\Psi^1=f$, then, when restricted to the nonseparated leaves $\Delta$ and $\Delta'$,  $\Phi^t=\Psi^t$ for each real number $t$.
\end{theorem}

We have already proven that $\folR$-flowable maps satisfy the four and eight point matching properties; in Section~\ref{S: abstract}--\ref{S: mainProof} we prove the converse implication. In Section~\ref{S: examples} we provide   an example of an $\folR$-homeomorphism that satisfies the four point matching property, but does not satisfy the eight point matching property, which means that, by Theorem~\ref{T: main}, it is not $\folR$-flowable. We also briefly describe an $\folR$-homeomorphism that does not satisfy the four point matching property.

%%%%%%%%%%%%%%%%%%%%%%%%%%%%%%%%%%%%%%%%%%%%%%%%%%%
\section{Flows on the real line}\label{S: abstract}
%%%%%%%%%%%%%%%%%%%%%%%%%%%%%%%%%%%%%%%%%%%%%%%%%%%

To prove Theorem~\ref{T: main} we define  flows on the nonseparated leaves $\Delta$ and $\Delta'$, and then extend these flows to a single flow on the whole plane. Since $\Delta$ and $\Delta'$ are each homeomorphic to the real line, it is instructive to begin by considering flows on $\mathbb{R}$. Each fixed point free order preserving homeomorphism $f$ of $\mathbb{R}$ embeds in uncountably many different flows. In order to pin down a particular flow whose time one map is $f$, we need more information about the flow. To this end, let $(\Phi^t)$ be a flow on $\mathbb{R}$ with $\Phi^1=f$, and such that $\Phi^t$ has no fixed points for each $t\neq 0$. This flow induces an equivalence relation $\equiv$ on $\mathbb{R}^2$
given by the property that $(x_1,x_2)\equiv(y_1,y_2)$ if and only if there is a real number $t$ such that $y_1=\Phi^t(x_1)$ and $y_2=\Phi^t(x_2)$. In other words, the equivalence classes are the orbits of the flow $\Phi^t \times \Phi^t$ acting on $\mathbb{R}^2$. (These equivalence classes  should not be confused with the Reeb foliation associated with Theorem~\ref{T: main}.) We say that $(\Phi^t)$ \emph{generates}  $\equiv$. Each equivalence class of $\equiv$ is the graph of some order preserving homeomorphism from $\mathbb{R}$ to $\mathbb{R}$. Furthermore, $\equiv$ satisfies several properties, three of which are listed below.
\begin{enumerate}
\item[(i)] If $(x_1,x_2)\equiv(y_1,y_2)$ then $(x_2,x_1)\equiv(y_2,y_1)$.
\item[(ii)] If $(x_1,x_2)\equiv(y_1,y_2)$ and $(x_2,x_3)\equiv(y_2,y_3)$ then $(x_1,x_3)\equiv(y_1,y_3 )$.
\item[(iii)] For all points $x$ and $y$ we have $(x,y)\equiv (f(x),f(y))$ and $(x,f(x))\equiv (y,f(y))$.
\end{enumerate}
Now we wish to show that, conversely, conditions (i), (ii), and (iii) are sufficient for an equivalence relation to be generated by a flow whose time one map is $f$. Suppose then that, as before, $f$ is a fixed point free order preserving homeomorphism of $\mathbb{R}$. We assume the existence of an equivalence relation $\equiv$ on $\mathbb{R}^2$, such that each equivalence class of $\equiv$ is the graph of some order preserving homeomorphism of $\mathbb{R}$, and such that conditions (i), (ii), and (iii) hold. We describe such a pair $(\equiv,f)$ as a \emph{flowable pair}. We use this terminology because of the following proposition, which is the main result of this section.

\begin{proposition}\label{P: real}
To each flowable pair $(\equiv,f)$ there corresponds a unique flow $(\Phi^t)$ on $\mathbb{R}$ that generates $\equiv$ and satisfies $\Phi^1=f$.
\end{proposition}

In order to prove Proposition~\ref{P: real}, we begin by considering only an equivalence relation $\equiv$ such that each equivalence  class is the graph of an order preserving homeomorphism of $\mathbb{R}$. That is, we ignore conditions (i), (ii), and (iii), and, in particular, for the moment we ignore the map $f$.
Let us consider the partition of $\mathbb{R}^2$ induced by $\equiv$. Clearly every vertical line (and also every horizontal line) intersects each equivalence class exactly once. From this and the Intermediate Value Theorem, it is easy to prove that this partition  is a  trivial foliation. We denote it by $\folF$.
We can define a map $\phi$ from $\mathbb{R}^3$ to $\mathbb{R}$ by the condition that $\phi(x,x',y)=y'$ if and only if $y'$ is the unique element of $\mathbb{R}$ such that $(x,x')\equiv (y,y')$. Then  the maps $u\mapsto\phi(x,x',u)$, $u\mapsto\phi(x,u,y)$ and $u\mapsto\phi(u,x',y)$ have the following interpretation. The graph of the first one is the equivalence class of $(x,x')$, thus by hypothesis it is an order preserving homeomorphism. If we identify in a natural way the vertical lines $\{x \} \times \mathbb{R}$ and $\{y \} \times \mathbb{R}$ with $\mathbb{R}$, then the second one is the holonomy of the foliation $\folF$; that is, the map obtained by following the leaves of $\folF$ from one line to the other one. It is clear (and well known) that such a map is also an order preserving homeomorphism. Similarly, the third map is identified with the holonomy from a horizontal line to a vertical line, and thus is an order reversing homeomorphism. These properties are summed up in the following lemma.

\begin{lemma}\label{L: phi}
Given points $x$, $x'$, and $y$ in $\mathbb{R}$, the maps $u\mapsto\phi(x,x',u)$ and $u\mapsto\phi(x,u,y)$ are order preserving homeomorphisms of $\mathbb{R}$, and the map $u\mapsto\phi(u,x',y)$ is an order reversing homeomorphism of $\mathbb{R}$.
\end{lemma}

We give two corollaries of Lemma~\ref{L: phi}.

\begin{corollary}\label{C: cts}
The map $\phi:\mathbb{R}^3\to\mathbb{R}$ is continuous.
\end{corollary}
\begin{proof}
It is well known (see, for example, \cite{KrDe1969}) that a map from $\mathbb{R}^n$ to $\mathbb{R}$ that is a homeomorphism in each argument is continuous.
\end{proof}

\begin{corollary}\label{C: sqrt}
Given two points  $a$ and $b$ in $\mathbb{R}$ with $a<b$ there is a unique real number $u$ such that $(a,u)\equiv (u,b)$.
\end{corollary}
\begin{proof}
The point  $u$ is the unique fixed point of the order reversing map $u\mapsto \phi(u,b,a)$.
\end{proof}

Now we suppose that $(\equiv,f)$ is a flowable pair. Notice that if $(a,b)\equiv (c,d)$ then, since $(b,f(b))\equiv (d,f(d))$ by condition (iii),  we have $(a,f(b))\equiv (c,f(d))$ by condition (ii).

A \emph{square-root of $f$} is a homeomorphism $s$ of $\mathbb{R}$ such that $s^2 = f$.

\begin{lemma}\label{L: flowable pair}
If $(\equiv,f)$ is a flowable pair then $f$ has a unique square-root $s$ such that $(\equiv,s)$ is also a flowable pair.
\end{lemma}
\begin{proof}
If $s$ is a square-root of $f$ and $(\equiv,s)$ is a flowable pair, then
\[
(x,s(x))\equiv (s(x),s^2(x))=(s(x),f(x)),
\]
 therefore $s$ is uniquely defined, by Corollary~\ref{C: sqrt}. We must show that such a map $s$ exists. Using Corollary~\ref{C: sqrt} we define $s(x)$ to be the unique real number such that $(x,s(x))\equiv (s(x),f(x))$. That is, $s(x)$ is the unique fixed point of the order reversing map $\gamma_x:u\mapsto \phi(u,f(x),x)$. To see that $s$ is continuous, let $\textnormal{Homeo}^-(\mathbb{R})$ denote the topological space of order reversing homeomorphisms of $\mathbb{R}$ equipped with the compact-open topology. Then $s$ is obtained by first applying the continuous map  $x\mapsto \gamma_x$ from $\mathbb{R}$ to $\textnormal{Homeo}^-(\mathbb{R})$ , and then applying the continuous map from  $\textnormal{Homeo}^-(\mathbb{R})$ to $\mathbb{R}$ which takes an order reversing homeomorphism to its fixed point.

Next we show that $s^2=f$. It follows from condition (i) that  $(s(x),x)\equiv (f(x),s(x))$, and by the observation made just before Lemma~\ref{L: flowable pair}, we have $(s(x),f(x))\equiv(f(x),fs(x))$. Hence, by considering the definition of $s(s(x))$ and the uniqueness part of Corollary~\ref{C: sqrt}, we see that
 $f(x)=s^2(x)$. Thus $s$ is an order preserving homeomorphism of the real line whose square is $f$.

It remains to verify condition (iii) for $s$. First we show that $(x,s(x))\equiv(y,s(y))$. Define $u$ to be the unique point such that $(x,s(x))\equiv (y,u)$, and define $v$ to be the unique point such that $(s(x),f(x))\equiv (u,v)$. Then, by condition (ii), we have $(x,f(x))\equiv (y,v)$; hence $v=f(y)$. Therefore
\[
(y,u)\equiv (x,s(x))\equiv (s(x),f(x))\equiv (u,f(y)).
\]
 By Corollary~\ref{C: sqrt}, this means that $u=s(y)$, as required.

Last we show that $(s(x),s(y))\equiv(x,y)$. Consider the order preserving homeomorphism $\delta_y:\mathbb{R}\to\mathbb{R}$ given by $\delta_y(u)=\phi(s(y),s(u),y)$. Let $v=\delta_y(x)$ and let $w=\delta_y(v)$. Then $(s(y),s(x))\equiv (y,v)$ and $(s(y),s(v))\equiv(y,w)$. Using property (i)  and bearing in mind the comment stated just before Lemma~\ref{L: flowable pair}, we can apply $s$ to each argument of each side of the first of the two equivalence relations to see that $(f(y),f(x))\equiv (s(y),s(v))$. By transitivity of $\equiv$, this means that $(f(y),f(x))\equiv (y,w)$. Hence, by condition (iii) for $f$, we have $(y,w)\equiv (y,x)$. Therefore  $w=x$. This means that $\delta_y^2$ is the identity map. But $\delta_y$ preserves order, therefore $\delta_y$ is the identity map. Therefore $v=x$, and $(s(x),s(y))\equiv (x,y)$.
\end{proof}

By applying Lemma~\ref{L: flowable pair} repeatedly we obtain a sequence $(f_k)$ of order preserving homeomorphisms such that $f_0=f$, $f_k=f_{k+1}^2$, and $(\equiv, f_k)$ is a flowable pair for each positive integer $k$. We can define a partial order $<$ on the group $\textup{Homeo}^+(\mathbb{R})$ of order preserving homeomorphisms of the real line, given by $g<h$ if and only if $g(x)<h(x)$ for each real number $x$. With respect to this partial order, the sequence $(f_k)$ is strictly monotonic. Moreover, we have the following result.

\begin{lemma}\label{L: f_k}
The sequence $(f_{k})$ converges uniformly on compact subsets of $\mathbb{R}$ to the identity.
\end{lemma}

It is convenient to assume henceforth that $f(x)>x$ for each real number $x$. The alternative is that $f(x)<x$ for each real number $x$, and in this case we can revert to the former situation by working with $f^{-1}$, rather than $f$. Note that this implies that for every positive integer $k$ and every real number $x$, $f_{k}(x) > x$.

\begin{proof}
The topology of locally uniform convergence coincides with the topology of pointwise convergence on $\textnormal{Homeo}^+(\mathbb{R})$, thus it suffices to prove pointwise convergence of $(f_k)$. Choose a real number $x$. The sequence $(f_k(x))$ is decreasing and bounded below by $x$; let $y$ be the limit of the sequence. By Corollary~\ref{C: sqrt} we have that $(x,f_k(x))\equiv (f_k(x),f_{k-1}(x))$ for each positive integer $k$. By Corollary~\ref{C: cts} we can take limits to obtain $(x,y)\equiv (y,y)$. Therefore $y=x$.
\end{proof}

Given a dyadic rational number $t$ we define
\[
\Phi^{t} = f_{k}^{\varepsilon_{k}}\circ  \cdots \circ f_{0}^{\varepsilon_{0}}
\]
where
\[
t = \varepsilon_{0} + \frac{\varepsilon_1}{2}+ \cdots + \frac{\varepsilon_{k}}{2^k}, \ \  \varepsilon_0\in\mathbb{Z}, \  \ \varepsilon_{i} \in \{0,1\}.
\]
Thus
\[
\Phi^t = f_{k}^{\varepsilon_{k} + 2\varepsilon_{k-1}+ \cdots +2^k\varepsilon_{0}}.
\]

In future we omit the symbol $\circ$ whenever we compose two functions.
The additive property $\Phi^{s+t}=\Phi^s\Phi^t$ holds whenever $s$ and $t$ are dyadic numbers. The map $t\mapsto \Phi^t$ preserves order, because if $p/2^k$ and $q/2^k$ are dyadic numbers such that $p<q$ then for each real number $x$ we have
\[
\Phi^{q/2^k}(x) = f_k^{q-p}\Phi^{p/2^k}(x) > \Phi^{p/2^k}(x).
\]

For the next lemma, recall that the set $\textup{Homeo}^+(\mathbb{R})$ can be equipped with the topology of uniform convergence on compact sets. This  topology is generated by the family of pseudo-distances $d_{K}$ given by
\[
d_{K}(g,h) = \textnormal{sup}_{x \in [-K,K]} \left( \mid g(x)-h(x)\mid + \mid g^{-1}(x)-h^{-1}(x)\mid \right)
\]
where $[-K,K]$ is a compact subset of $\mathbb{R}$.
These pseudo-distances give rise to a distance for which $\textup{Homeo}^+(\mathbb{R})$ becomes a complete metric space. The set of dyadic numbers is  equipped with the usual distance.
\begin{lemma}
The map $t \mapsto \Phi^{t}$ from the group of dyadic numbers to $\textup{Homeo}^+(\mathbb{R})$ is uniformly continuous on bounded sets.
\end{lemma}
\begin{proof}
Let $M$ be a positive dyadic number. We shall prove uniform continuity on the  interval $[-M,M]$. For this we choose $\varepsilon>0$ and a compact set $[-K,K]$, and we shall determine a positive integer $k$ such that for points $t_1$ and $t_2$ in $[-M,M]$ with $0 < t_2-t_1 < 2^{-k}$, we have $d_{K}(\Phi^{t_1},\Phi^{t_2}) < \varepsilon$.

Recall that  $t \mapsto \Phi^t$ preserves the partial order on $\textup{Homeo}^+(\mathbb{R})$. Thus for any $t$ in $[-M,M]$ and any $x$ in $[-K,K]$, the point
$\Phi^t(x)$ belongs to the interval $[\Phi^{-M}(-K), \Phi^{M}(K)]$.
Applying Lemma~\ref{L: f_k}, we can choose $k>0$ such that, for every $y$ in $[\Phi^{-M}(-K), \Phi^{M}(K)]$, $y < f_{k}(y) < y + \varepsilon$ and $ y-\varepsilon <  f_{k}^{-1}(y) < y$.
Furthermore, for every $t_1$ and $t_2$ with $0 < t_2-t_1 < 2^{-k}$ we have $\Phi^{t_2-t_1} < f_{k}$. Thus, for $x$ in $[-K,K]$,
\[
\Phi^{t_2}(x)-\Phi^{t_1}(x) = \Phi^{t_2-t_1}\left(\Phi^{t_1}(x)\right) - \Phi^{t_1}(x) < f^k\left(\Phi^{t_1}(x)\right) - \Phi^{t_1}(x) < \varepsilon.
\]
Similarly $|\Phi^{-t_2}(x)-\Phi^{-t_1}(x)|<\varepsilon$. Hence $d_K(\Phi^{t_1},\Phi^{t_2})<\varepsilon$, as required.
\end{proof}

Finally, we can prove Proposition~\ref{P: real}.

\begin{proof}[Proof of Proposition~\ref{P: real}]
The map $t \mapsto \Phi^{t}$ is uniformly continuous when restricted to bounded subsets so it extends uniquely to a continuous map from $\mathbb{R}$ to $\textnormal{Homeo}^+(\mathbb{R})$. Since the equation $\Phi^{s+t}=\Phi^{s}  \Phi^{t}$ is satisfied whenever $s$ and $t$ are dyadic numbers, by continuity, it is also satisfied whenever $s$ and $t$ are any real numbers. Thus $(\Phi^t)$ is a flow on $\mathbb{R}$ such that $\Phi^1=f$.

We must show that $(\Phi^t)$ generates $\equiv$. By condition (iii) we have that $(\Phi^t(x),\Phi^t(y))\equiv (x,y)$ for each dyadic number $t$ and each pair of points $x$ and $y$. By continuity, and since the equivalence classes are closed,
 this relation holds for all real numbers $t$. This shows that the leaves of the foliation generated by $(\Phi^t)$ coincide with the equivalence classes of $\equiv$; therefore $(\Phi^t)$ generates $\equiv$.

We finish by proving the uniqueness part of the proposition; that is, we show that there is only one flow $(\Phi^t)$ that generates $\equiv$ and satisfies $\Phi^1=f$. Given a flow  $(\Phi^t)$ with these properties, we know that $(\equiv,\Phi^t)$ is a flowable pair for each real number $t$. By Lemma~\ref{L: flowable pair}, this uniquely specifies $\Phi^{2^{-n}}$ for each positive integer $n$. Since the dyadic numbers are dense in $\mathbb{R}$, we deduce by additivity and continuity that $\Phi^t$ is uniquely defined for each real number $t$.
\end{proof}

%%%%%%%%%%%%%%%%%%%%%%%%%%%%%%%%%%%%%%%%%%%%%%%%%%%%%%%%%%%%%%%%%
\section{The equivalence relation $\sim_f$}\label{S: DeltaDelta'}
%%%%%%%%%%%%%%%%%%%%%%%%%%%%%%%%%%%%%%%%%%%%%%%%%%%%%%%%%%%%%%%%%

We return to considering homeomorphisms of the plane. Throughout this section we assume that $f$ is an $\folR$-homeomorphism that satisfies the four point matching  property. In the introduction we described how to define a relation $\sim_f$ on $\Delta\times\Delta'$, and in this section we develop the properties of $\sim_f$.

Recall that, in order to define $\sim_f$, we chose a pair $(x,x')$ in $\Delta\times\Delta'$ and a sequence $(x_k)$ in $\mathbb{R}^2$ such that $x_k\to x$ and $f^{k}(x_k)\to x'$. We first prove that, given the pair $(x,x')$, such sequences exist.

\begin{lemma}\label{L: sequences}
Given a pair $(x,x')$ in $\Delta\times\Delta'$ there exists a sequence $(x_k)$ in $\mathbb{R}^2$ such that $x_k\to x$ and $f^{k}(x_k)\to x'$.
\end{lemma}
\begin{proof}
Choose a chart $\varphi$ that maps a neighbourhood $N$ of $x$ to $]0,1[^2$, and maps the foliation of $N$ by $\folR$ to the horizontal foliation of $]0,1[^2$. Let $\delta$ be a vertical open segment in $]0,1[^2$ with one end point at $\varphi(x)$. By reflecting $\delta$ in the horizontal line through $\varphi(x)$ if necessary, we can assume that the open segment $\gamma=\varphi^{-1}(\delta)$ lies in the middle component $U_2$ of the Reeb foliation. The segment $\gamma$ has one end point at $x$, and lies transverse to the leaves of $\folR$. Likewise there is an open line segment $\gamma'$ in $U_2$ that lands at $x'$ and lies transverse to $\folR$. By truncating $\gamma$ or $\gamma'$ we can assume that the end points $u$ and $v$ of $\gamma$ and $\gamma'$ that lie in $U_2$ both lie on the same leaf $L$. Let $V$ denote the connected component of $U_2\setminus L$ that contains $\gamma$ and $\gamma'$.

Choose a positive integer $m$ such that $f^m(u)>v$. For each integer $k>m$, the line segment $f^k(\gamma)$ lies wholly in $V$, and it has one end point on $\Delta$, and the other end point on $L$. Since $\gamma'$ disconnects $V$ into two components, and there are points of $f^k(\gamma)$ in each of these components, we deduce the existence of a point $x_k$ in $\gamma$ such that $f^k(x_k)\in \gamma'$.

It remains to show that $x_k\to x$ and $f^k(x_k)\to x'$. Observe that, for each compact subset $K$ of $U_2$, both sequences $K,f(K),f^2(K),\ldots$ and $K,f^{-1}(K),f^{-2}(K),\ldots$ accumulate only at $\infty$ (that is, they have no accumulation points in $\Delta\cup\Delta'\cup U_2$). This means that neither sequence $(x_k)$ nor $(f^k(x_k))$ can accumulate within $U_2$, as required.
\end{proof}

\begin{lemma}\label{L: leaf}
Given  a pair $(x,x')$ in $\Delta\times\Delta'$ and a point $y$ in $\Delta$ there exists a unique point $y'$ in $\Delta'$ such that $(x,x')\sim_f (y,y')$.
\end{lemma}

By exchanging the r\^oles of $\Delta$ and $\Delta'$, we see that it is also true that given the pair $(x,x')$ and a point $y'$ in $\Delta'$ there exists a unique point $y$ in $\Delta$ such that $(x,x')\sim_f (y,y')$.

\begin{proof}
Choose a sequence $(x_k)$ in $\mathbb{R}^2$ such that $x_k\to x$ and $f^{k}(x_k)\to x'$. Given any neighbourhood $N$ of the point $y$, all but finitely many of the leaves $L(x_k)$ intersect $N$, therefore we may choose a sequence $(y_k)$ in $\mathbb{R}^2$ such that $y_k\in L(x_k)$ and $y_k\to y$. (In future we often have to assert the existence of such a sequence $(y_k)$, and we do so without elaboration.) By the four point matching property we conclude that the sequence $(f^k(y_k))$ converges to a point $y'$. Hence $(x,x')\sim_f (y,y')$.

It remains to prove that this point $y'$ is unique. To this end, choose another point $z'$
in $\Delta'$ such that $(x,x')\sim_f (y,z')$. Then there are sequences $(u_k)$ and $(v_k)$ in $\mathbb{R}^2$ such that $v_k\in L(u_k)$, and $u_k\to x$, $f^{k}(u_k)\to x'$, $v_k\to y$, and $f^{k}(v_k)\to z'$.  We can apply the four point matching property to the spliced sequences $x_1,u_2,x_3,u_4,\ldots$ and $y_1,v_2,y_3,v_4,\ldots$ to deduce that the sequence $f^{1}(y_1),f^{2}(v_2),f^{3}(y_3),f^{4}(v_4),\ldots$ converges to a single value in $\Delta'$. Since $(f^{2k-1}(y_{2k-1}))$ and $(f^{2k}(v_{2k}))$ are both subsequences of this final sequence we conclude that $y'=z'$.
\end{proof}

\begin{lemma}
The relation $\sim_f$ is an equivalence relation on $\Delta\times\Delta'$.
\end{lemma}
\begin{proof}
That $\sim_f$ is reflexive follows immediately from Lemma~\ref{L: sequences}, and that $\sim_f$ is symmetric follows immediately from its definition. It remains to show that $\sim_f$ is transitive, and to this end we suppose that  $(x,x')\sim_f (y,y')$, and $(y,y')\sim_f (z,z')$. As usual we choose sequences $(x_k)$ and $(y_k)$ in $\mathbb{R}^2$ such that $y_k\in L(x_k)$ and  $x_k\to x$, $f^{k}(x_k)\to x'$, $y_k\to y$, and $f^{k}(y_k)\to y'$. Let $(z_k)$ be a sequence in $\mathbb{R}^2$ such that $z_k\in L(x_k)$ and $z_k\to z$ (recall that $L(x_k)=L(y_{k})$).
 By the four point matching property, the sequence $(f^{k}(z_k))$ converges to a point in $\Delta'$, which, by Lemma~\ref{L: leaf}, must be $z'$.
 Therefore $(x,x')\sim_f(z,z')$.
\end{proof}

Given a pair $(x,x')\in\Delta\times\Delta'$, we may apply  Lemma~\ref{L: leaf}  to define a map $\alpha_{xx'}:\Delta\to\Delta'$ by the equation $\alpha_{xx'}(y)=y'$, where $y'$ is the unique member of $\Delta'$ such that $(y,y')\sim_f(x,x')$. An application of the comment following the statement of Lemma~\ref{L: leaf}  shows that this map is a bijection.

\begin{lemma}\label{L: alpha}
The map $\alpha_{xx'}$ is an order preserving homeomorphism from $\Delta$ to $\Delta'$.
\end{lemma}
\begin{proof}
Because $\alpha_{xx'}$ is a bijection, we need only show that it preserves order. Choose points $y^1$ and $y^2$ in $\Delta$ such that $y^1<y^2$. Let $(x_k)$ be a sequence in $\mathbb{R}^2$ such that $x_k\to x$ and $f^{k}(x_k)\to x'$. For $j=1,2$ choose a sequence $(y^j_k)$ such that $y^j_k\in L(x_k)$ and $y^j_k\to y_j$. By the four point matching property, $(f^{k}(y^j_k))$ converges to a point $y'^j_k$, which is $\alpha_{xx'}(y^j)$. For large enough integers $k$, we have $y^1_k<y^2_k$ on $L(x_{k})$.
 Since $f$ preserves order on $L(x_k)$ we also have that $f^{k}(y^1_k)<f^{k}(y^2_k)$. Therefore either $y'^1<y'^2$ or $y'^1=y'^2$, and in fact the latter case cannot happen by the uniqueness part of Lemma~\ref{L: leaf}.
\end{proof}

Given two equivalent pairs $(x_1,x'_1)$ and $(x_2,x'_2)$, the maps $\alpha_{x_1x'_1}$ and $\alpha_{x_2x'_2}$ are identical. Indeed  the graph of $\alpha_{xx'}$ is  the equivalence class  of the pair $(x,x')$ in   $\Delta\times\Delta'$.
 Let $\alpha: \Delta\times\Delta'\times\Delta\to\Delta'$ be the map given by $\alpha(x,x',y)=y'$ if and only if $(x,x')\sim_f(y,y')$. Since $\Delta$ and $\Delta'$ are each copies of $\mathbb{R}$, we may invoke Lemma~\ref{L: phi}
  to deduce that $\alpha$  is a homeomorphism in each of its arguments separately.

We record a few basic properties of the interaction between $f$ and $\sim_f$ in the next lemma.

\begin{lemma}\label{L: f}
Given two equivalent pairs $(x,x')$ and $(y,y')$ in $\Delta\times\Delta'$ we have $(x,x')\sim_f (f(x),f(x'))$, $(x,f(x'))\sim_f (y,f(y'))$, and $(f(x),x')\sim_f (f(y),y')$.
\end{lemma}
\begin{proof}
The proofs of these three equivalences are elementary, and we supply only the first. By Lemma~\ref{L: sequences} there is a sequence $(x_k)$  in $\mathbb{R}^2$   such that $x_k\to x$ and $f^{k}(x_k)\to x'$. Define $y_k=f(x_k)$, then the sequence $(y_k)$ satisfies $y_k\to f(x)$ and $f^{k}(y_k)\to f(x')$. Therefore $(x,x')\sim_f (f(x),f(x'))$.
\end{proof}

%%%%%%%%%%%%%%%%%%%%%%%%%%%%%%%%%%%%%%%%%%%%%%%%%%%%%%%%%%%%%%%%%%%%
\section{The equivalence relation $\approx_f$ }\label{S: DeltaDelta}
%%%%%%%%%%%%%%%%%%%%%%%%%%%%%%%%%%%%%%%%%%%%%%%%%%%%%%%%%%%%%%%%%%%%

We assume throughout this section that $f$ is an $\folR$-homeomorphism that satisfies both the four and eight point matching properties. Recall the definition of $\approx_f$ on $\Delta\times\Delta$: $(x_1,y_1)\approx_f (x_2,y_2)$ if and only if there are points $x'$ and $y'$ in $\Delta'$ such that $(x_1,x')\sim_f (y_1,y')$ and $(x_2,x')\sim_f (y_2,y')$. The point $x'$ in this definition can be chosen to be any point on $\Delta'$. To see this, choose any element $z'$
 of  $\Delta'$. Then, by Lemma~\ref{L: leaf}, there is an element $v'$ of $\Delta'$ such that $(x_1,z')\sim_f (y_1,v')$.
It follows from the eight point matching property that $(x_2,z')\sim_f (y_2,v')$.

\begin{lemma}\label{L: bijection}
Given points $x_1$, $y_1$, and $x_2$ in $\Delta$, there is a unique point $y_2$ in $\Delta$ such that $(x_1,y_1)\approx_f (x_2,y_2)$.
\end{lemma}
\begin{proof}
Choose an arbitrary point $x'$ on $\Delta'$. Then by Lemma~\ref{L: leaf} there is a point $y'$ on $\Delta'$ such that $(x_1,x')\sim_f (y_1,y')$. Apply Lemma~\ref{L: leaf} again---this time with the r\^oles of $\Delta$ and $\Delta'$ reversed---to deduce the existence of a point $y_2$ in $\Delta$ such that $(x_2,x')\sim_f (y_2,y')$. Therefore $(x_1,y_1)\approx_f (x_2,y_2)$. Suppose there is another point $y^*_2$ in $\Delta$ such that $(x_1,y_1)\approx_f (x_2,y^*_2)$. Then there is a point $z'$
 in $\Delta'$ such that $(x_1,x')\sim_f (y_1,z')$ and $(x_2,x')\sim_f (y^*_2,z')$. However, we know that $(x_1,x')\sim_f (y_1,y')$; thus by Lemma~\ref{L: leaf}
  we see that $y'=z'$. Apply Lemma~\ref{L: leaf} again, this time to the relations $(x_2,x')\sim_f (y_2,y')$ and $(x_2,x')\sim_f (y^*_2,z')$, to conclude that $y^*_2=y_2$.
\end{proof}

\begin{lemma}\label{L: equiv}
The relation $\approx_f$ is an equivalence relation on $\Delta\times\Delta$.
\end{lemma}
\begin{proof}
That $\approx_f$ is reflexive follows immediately from the four point matching property (with any choice of $x'$), and that $\approx_f$ is symmetric follows immediately from the definition of $\approx_f$. It remains to show that $\approx_f$ is transitive, and to this end we suppose that $(x_1,y_1)\approx_f (x_2,y_2)$ and $(x_2,y_2)\approx_f (x_3,y_3)$. Then there are points $x'_1$, $y'_1$, $x'_2$, and $y'_2$ in $\Delta'$ such that $(x_1,x'_1)\sim_f (y_1,y'_1)$ and
$(x_2,x'_1)\sim_f (y_2,y'_1)$, and $(x_2,x'_2)\sim_f (y_2,y'_2)$ and
$(x_3,x'_2)\sim_f (y_3,y'_2)$. We apply the eight point matching property to
the first three of these $\sim_f$ relations to  see that $(x_1,x'_2)\sim_f (y_1,y'_2)$. This relation coupled with the fourth of the four listed $\sim_f$ relations shows that $(x_1,y_1)\approx_f (x_3,y_3)$.
\end{proof}

Given two points $x_1$ and $x_2$ in $\Delta$ we can use Lemma~\ref{L: bijection} to construct a bijection $\beta_{x_1x_2}:\Delta\to\Delta$ given by $\beta_{x_1x_2}(y_1)=y_2$ if and only if $(x_1,y_1)\approx_f(x_2,y_2)$.

\begin{lemma}\label{L: op}
The map $\beta_{x_1x_2}$ is an order preserving homeomorphism of $\Delta$.
\end{lemma}
\begin{proof}
Recall the definition of $\alpha_{xx'}$ from the previous section. Given any point $x'$ in $\Delta'$ we have $\beta_{x_1x_2} = \alpha_{x_2x'}^{-1}\alpha_{x_1x'}$, and the result follows immediately from Lemma~\ref{L: alpha}.
\end{proof}

Since $\Delta$ is homeomorphic to $\mathbb{R}$ we may ask whether $(\approx_f,{f|}_\Delta)$ is a flowable pair.

\begin{lemma}\label{L: pair}
The quantity $(\approx_f,{f|}_\Delta)$ is a flowable pair.
\end{lemma}
\begin{proof}
By Lemma~\ref{L: equiv}, $\approx_f$ is an equivalence relation on $\Delta\times\Delta$, and by Lemma~\ref{L: op} each equivalence class is the graph of an order preserving homeomorphism from $\Delta$ to $\Delta$.  Each of the three conditions of Section~\ref{S: abstract} which define a flowable pair follow from properties of $\sim_f$. For example, (i) follows from symmetry of $\sim_f$, because if $(x_1,x_2)\approx_f (y_1,y_2)$ then there is a pair $(x',y')$  in $\Delta'\times\Delta'$ such that $(x_1,x')\sim_f(x_2,y')$ and $(y_1,x')\sim_f(y_2,y')$. Therefore $(x_2,y')\sim_f(x_1,x')$ and $(y_2,y')\sim_f(y_1,x')$, which means that $(x_2,x_1)\approx_f (y_2,y_1)$. Likewise condition (ii) follows immediately from transitivity of $\sim_f$, and condition (iii) follows immediately from the properties  established in Lemma~\ref{L: f}.
\end{proof}

We can now apply Proposition~\ref{P: real} to $(\approx_f,{f|}_\Delta)$. We state the outcome for clarity and convenience.

\begin{proposition}\label{P: flow}
There is a unique flow $(\Phi^t)$ on $\Delta$ that generates $\approx_f$ and satisfies $\Phi^1={f|}_\Delta$.
\end{proposition}

%%%%%%%%%%%%%%%%%%%%%%%%%%%%%%%%%%%%%%%%%%%%%%%%%%%%%%%%%%%%
\section{Proof of Theorem~\ref{T: main}}\label{S: mainProof}
%%%%%%%%%%%%%%%%%%%%%%%%%%%%%%%%%%%%%%%%%%%%%%%%%%%%%%%%%%%%

Suppose that $f$ is an $\folR$-homeomorphism that satisfies both the four and eight point matching properties. From Proposition~\ref{P: flow} we know that there is a unique flow $(\Phi^t)$ on $\Delta$ that generates $\approx_f$, and is such that $\Phi^1={f|}_\Delta$.
By symmetry of the roles played by $\Delta$ and $\Delta'$, we can also define an equivalence relation on $\Delta'\times\Delta'$, which we still denote by $\approx_f$, by requiring that
$(x'_1,y'_1)\approx_f(x'_2,y'_2)$ if and only if there is a pair $(x,y)$ in $\Delta\times \Delta$ such that both $(x,x'_{1})\sim_f (y,y'_{1})$ and $(x,x'_{2})\sim_f (y,y'_{2})$.
Results from Section~\ref{S: DeltaDelta}, and in particular Proposition~\ref{P: flow}, apply to the relation $\approx_f$ on $\Delta'\times\Delta'$. Thus we can extend the flow $(\Phi^t)$ from $\Delta$ to $\Delta\cup\Delta'$ in such a way that the restriction of $(\Phi^t)$ to $\Delta'$ is the unique flow on $\Delta'$ that generates $\approx_f$, and is such that $\Phi^1|_{\Delta'}={f|}_{\Delta'}$.

\begin{lemma}\label{L: boundary continuity}
For each pair $(x,x')$ in $\Delta\times\Delta'$ and each real number $t$ we have
\[
\left(x,x'\right)\sim_f\left(\Phi^t(x),\Phi^t(x')\right).
\]
\end{lemma}
\begin{proof}
Choose a point $x'$ in $\Delta'$ and a real number $t$. For any pair $(x,y)$ in $\Delta\times\Delta$ we have $(x,y)\approx_f (\Phi^t(x),\Phi^t(y))$. Therefore there is a unique point $z$ in $\Delta'$ such that
\begin{equation}\label{E: Psi}
(x,x')\sim_f (\Phi^t(x),z),\quad (y,x')\sim_f (\Phi^t(y),z).
\end{equation}
We define a function $\Psi^t:\Delta'\to\Delta'$ by the equation $\Psi^t(x')=z$. In other words,
\[
\Psi^t(x')=\alpha(x,x',\Phi^t(x)),
\]
and this definition does not depend on the choice of $x$. Using properties of $\alpha$ we deduce that $\Psi^t$ is an order preserving homeomorphism of $\Delta'$ which has no fixed points whenever $t\neq 0$, and which coincides with the identity map when $t=0$. Moreover, using continuity of $\Phi$, we have that $\Psi$ is jointly continuous in $t$ and $x$. Given two real numbers $s$ and $t$ we have that
\[
(x,x')\sim_f (\Phi^t(x),\Psi^t(x'))\sim_f (\Phi^s\Phi^t(x),\Psi^s\Psi^t(x'))
\]
and
\[
(x,x')\sim_f  (\Phi^{s+t}(x),\Psi^{s+t}(x')).
\]
Since $\Phi^{s+t}(x)=\Phi^s\Phi^t(x)$ we conclude that $\Psi^{s+t}(x')=\Psi^s\Psi^t(x')$. Therefore $(\Psi^t)$ is a flow. We only have to
 show that $(\Psi^t)=(\Phi^t)$ because then the lemma follows from \eqref{E: Psi}.

We appeal to the uniqueness part of Proposition~\ref{P: flow} applied to $\Delta'$. First observe that
\[
\Psi^1(x')=\alpha(x,x',f(x))=f(x'),
\]
by Lemma~\ref{L: f}. Second observe that, given a pair $(x',y')$ in $\Delta'\times\Delta'$, we have, by \eqref{E: Psi},
\[
(x,x')\sim_f (\Phi^t(x),\Psi^t(x')),\quad (x,y')\sim_f (\Phi^t(x),\Psi^t(y')).
\]
Hence $(\Psi^t(x'),\Psi^t(y'))\approx_f (x',y')$. Therefore the leaves of $\Delta'\times\Delta'$ under $\approx_f$ coincide with the leaves of $\Delta'\times\Delta'$ resulting from the flow $(\Psi^t)$. In other words, $\approx_f$ is generated by $(\Psi^t)$. We can now invoke the uniqueness part of Proposition~\ref{P: flow} to see that $(\Psi^t)=(\Phi^t)$, as required.
\end{proof}

\begin{proof}[Proof of Theorem~\ref{T: main}]
To prove Theorem~\ref{T: main} we must extend  the flow $\Phi^t$ from $\Delta\cup\Delta'$
to the whole plane, as an $\folR$-flow with $\Phi^1=f$. Let $O=U_1\cup\Delta\cup U_2$ (refer to Figure~\ref{F: Reeb}).
 The restriction of $\folR$ to  $O$ is a trivial foliation, and we may choose an orientation preserving  homeomorphism $H:O\rightarrow \mathbb{R}^2$ such that (i) the leaves of $\folR$ are mapped to horizontal lines, (ii) $H(\Delta)=\mathbb{R}\times\{0\}$, (iii) $H\Phi^tH^{-1}(x,0)=(x+t,0)$
 for each $x$ in $\mathbb{R}$, and (iv) $Hf H^{-1}(x,y)=(x+1,y)$ for each $(x,y)$ in $\mathbb{R}^2$. Let $\tau_t$ denote the map $(x_1,x_2)\mapsto (x_1+t,x_2)$ on $\mathbb{R}^2$. We extend the domain of $\Phi^t$ to $O$ by defining $\Phi^t=H^{-1}\tau_t H$. This is a flow on $O$ whose leaves are the leaves of $\folR$, and which satisfies $\Phi^1(x)=f(x)$ for $x$ in $O$. Likewise we can extend the domain of $\Phi^t$ from $\Delta'$ to the closed set $O'=U_3\cup\Delta'$ in a such a way that $\Phi^1(x)=f(x)$ for $x$ in $O'$.

We have now defined a  collection of maps $(\Phi^t)$, each of which is a bijection of $\mathbb{R}^2$, and they satisfy the property that $\Phi^{s+t}(x)=\Phi^s\Phi^t(x)$ for each point $x$ in $\mathbb{R}^2$ and each pair of real numbers $s$ and $t$. Moreover, $\Phi$ is jointly continuous in $x$ and $t$ whenever $x$ lies outside $\Delta'$. Suppose that $\Phi$ is \emph{not} jointly continuous for points $x$ in $\Delta'$. Then there is a sequence $(x_k)$ in $U_2$ that converges to a point $x'$ in $\Delta'$, and a sequence $(t_k)$ in $\mathbb{R}$ that converges to a real number $t$, such that the sequence $\left(\Phi^{t_k}(x_k)\right)$ either converges to a point $y'$ on $\Delta'$ that is distinct from $\Phi^t(x')$, or else diverges to $\infty$. The second case cannot occur because we can choose $M>0$ such that, for each $k$, $-M< t_{k} <M$, which means that $f^{-M} ( x_{k} ) <  \Phi^{t_{k}}(x_{k}) < f^M(x_{k})$  on $L(x_{k})$, and hence $\left(\Phi^{t_{k}}(x_{k})\right)$ lies in a bounded subset of $U_2$. For the first case, we choose a sequence of integers $(n_k)$ such that $(f^{-n_k}(x_{n_k}))$ converges to a point $x$ on $\Delta$. By Lemma~\ref{L: sequences} we can choose a sequence $(y_k)$ in $\mathbb{R}^2$ such that $y_k\to x$, $f^k(y_k)\to x'$, and $y_{n_k} = f^{-n_k}(x_{n_k})$. Now, $\Phi^{t_k}(y_k)\in L(y_k)$ for each $k$, and $\Phi^{t_k}(y_k) \to \Phi^t(x)$ as $k\to\infty$, therefore, by the four point matching property, there is a point $z'$ on $\Delta'$ such that $f^k\left(\Phi^{t_k}(y_k)\right)\to z'$. The $n_k$th term from this last sequence is $\Phi^{t_{n_k}}(x_{n_k})$; hence $z'=y'$. Thus $(x,x')\sim_f (\Phi^t(x),y')$. By Lemma~\ref{L: boundary continuity} we also have
$(x,x')\sim_f (\Phi^t(x),\Phi^t(x'))$. This implies that $y'=\Phi^t(x')$, which is a contradiction. Therefore $\Phi$ is jointly continuous throughout $\mathbb{R}^2$.

The uniqueness part of Theorem~\ref{T: main} follows immediately from the uniqueness part of Proposition~\ref{P: flow}.
\end{proof}

%%%%%%%%%%%%%%%%%%%%%%%%%%%%%%%%%%%%%%%%%%%%%%%%%%%%%%%%%%%%%%%%%%%%%%%%%%%
\section{Necessity of the eight point matching property}\label{S: examples}
%%%%%%%%%%%%%%%%%%%%%%%%%%%%%%%%%%%%%%%%%%%%%%%%%%%%%%%%%%%%%%%%%%%%%%%%%%%

In this section we supply two examples. The first is an $\folR$-homeomorphism that satisfies the four point matching property, but does not satisfy the eight point matching property, which means that, by Theorem~\ref{T: main}, it is not $\folR$-flowable. We thereby establish the independence of the four and eight point matching properties, and prove that Theorem~\ref{T: main} fails in the absence of the eight point matching property. Our second example is an $\folR$-homeomorphism that does not satisfy the four point matching property.

To simplify computations in our first example, we map the middle section $U_2$ of Figure~\ref{F: Reeb} to the quadrant $Q=\{(x,y)\in\mathbb{R}^2\,:\, x> 0, y>0 \}$ by an orientation preserving homeomorphism, in such a way that the resulting foliation of $Q$, which we also denote by $\folR$, consists of the curves $xy=c$, for $c>0$. In this new model, $\Delta$ is the positive $y$-axis and $\Delta'$ is the positive $x$-axis. The foliation is oriented such that if $(x_1,y_1)$ and $(x_2,y_2)$ lie on the same leaf in $Q$ then $(x_1,y_1)<(x_2,y_2)$ if and only if $x_1<x_2$. The sections $U_1$ and $U_3$ of Figure~\ref{F: Reeb} play no role in the four or eight point matching properties, so their omission from this new model is insignificant.

For the construction, we need an auxiliary function $\beta$ provided by the following lemma.
\begin{lemma}\label{B}
There exists a function $\beta: [1/2,2]\to\mathbb{R}$ with the following properties:
\begin{enumerate}
\item[(i)] $\beta(\theta)=1$ if $\theta\in \left[1/2,3/4\right]\cup\left[1,3/2\right]\cup\{2\}$, and elsewhere $\beta(\theta)>1$;
\item[(ii)] if $\theta\in [1/2,1]$ then $\beta(\theta)=\beta(2\theta)$;
\item[(iii)] $\theta/\beta(\theta)$ is a strictly increasing self homeomorphism of $\left[1/2,2\right]$;
\item[(iv)] $\theta/\beta(\theta)^2$ is a strictly increasing self homeomorphism of $\left[1/2,2\right]$.
\end{enumerate}
\end{lemma}
\begin{proof}
We choose a continuously differentiable homeomorphism $h:\left[1/2,1\right]\rightarrow\left[1/2,1\right]$  that satisfies $h(\theta)=\theta$ on $\left[1/2,3/4\right]\cup \{1\}$, $h(\theta)< \theta$ on $\left]3/4,1\right[$, and $h'(\theta)>3/4$ everywhere. Define $\beta:\left[1/2,1\right]\rightarrow [1,+\infty\, [$ by $\beta(\theta)=\theta/h(\theta)$. Observe that
\begin{align*}
\frac{d}{d\theta}\left(\frac{\theta}{\beta(\theta)^2}\right)&=\frac{d}{d\theta}\left(\frac{h(\theta)^2}{\theta}\right) \\ &= \frac{h(\theta)}{\theta^2}(2\theta h'(\theta)-h(\theta)) \\&> \frac{h(\theta)}{\theta^2}\left(\frac{3\theta}{2}-h(\theta)\right)>0.
\end{align*}
Therefore $\theta/\beta(\theta)^2$ is an increasing homeomorphism of $\left[1/2,1\right]$. To complete the construction, extend the domain of $\beta$ to $\left[1/2,2\right]$ by defining, for $\theta$ in $[1,2]$, $\beta(\theta)=\beta(\theta/2)$.
\end{proof}

Let $k$ be the self homeomorphism of the set
\[
S = \left\{(x,y)\in Q\,:\, \frac{1}{2}\leq\frac{y}{x}< 2\right\},
\]
 given by $k(x,y)=(\beta(y/x)x,y/\beta(y/x))$. That $k$ is a homeomorphism follows from Lemma~\ref{B}(iv). Define $g$ to be the self homeomorphism of $Q$ given by $(x,y)\mapsto (2x,y/2)$. Finally, define $f$ to be the self homeomorphism of $Q$ that is equal to $gk$ on $S$, and elsewhere $f$ coincides with $g$. We will show that $f$ satisfies the four point matching property, but not the eight point matching property.

First we show that $f$ satisfies the four point matching property. Choose a sequence $(a_n,b_n)$ in $Q$ that converges to $(0,b)$ on $\Delta$, and is such that $f^n(a_n,b_n)$ converges to $(a,0)$ on $\Delta'$. Let $p_n$ be the unique integer such that $(2^{p_n}a_n,b_n/2^{p_n})$ lies in $S$. In other words, $p_n$ is the unique integer such that the quantity $\theta_n=b_n/(a_n2^{2p_n})$ satisfies $1/2 \leq \theta_n <2$. Observe that, for sufficiently large $n$,
\begin{equation}\label{A}
f^n(a_n,b_n) = \left(2^na_n\beta(\theta_n),\tfrac{b_n}{2^n\beta(\theta_n)}\right).
\end{equation}

\begin{lemma}\label{L: beta}
There is an element $\theta_0$ of $[1/2,2]$ such that $\beta(\theta_n)\to\beta(\theta_0)$ as $n\to\infty$.
\end{lemma}
\begin{proof}
Let $\theta^*$ be a limit point of $(\theta_n)$. Then there is a subsequence $(\theta_{n_k})$ that converges to $\theta^*$. Since $(a_n,b_n)\to (0,b)$ and $f^n(a_n,b_n)\to (a,0)$,  we see that $b_n\to b$ and, by equation \eqref{B},   $2^na_n\beta(\theta_n)\to a$ as $n\to\infty$. Observe also that
\[
\frac{\theta_n}{\beta(\theta_n)} = 2^{n-2p_{n}}\frac{b_n}{2^na_n\beta(\theta_n)}.
\]
  It follows that the integral sequence $\left(2^{n_k-2p_{n_k}}\right)$ converges to a positive number. Hence $n_k-2p_{n_k}$ is eventually equal to some integer $m$, which means that $\theta^*/\beta(\theta^*)=2^mb/a$. Suppose that $b/a$ is not itself a power of $2$. There are precisely two numbers of the form $2^mb/a$ that lie in $[1/2,2]$; call them $v_0$ and $2v_0$.  By Lemma~\ref{B}(i) and (iii) we can choose $\theta_0$ in $[1/2,1]$ such that $\theta_0/\beta(\theta_0)=v_0$. By Lemma~\ref{B}(ii), $2\theta_0/\beta(2\theta_0)=2v_0$. Thus, again using Lemma~\ref{B}(iii), there are precisely two values of $\theta$ in $[1/2,2]$, namely $\theta_0$ and $2\theta_0$, such that $\theta/\beta(\theta)$ is equal to $b/a$ multiplied by a power of $2$. This means that $\theta_0$ and $2\theta_0$ are the only limit points of $(\theta_n)$, and since $\beta(\theta_0)=\beta(2\theta_0)$, we see that $\beta(\theta_n)\to\beta(\theta_0)$. In the special case when $b/a$ is a power of $2$, there are three possible limit points of $(\theta_n)$, namely $1/2$, $1$, and $2$, and since $\beta(1/2)=\beta(1)=\beta(2)$, we see that $\beta(\theta_n)\to\beta(1/2)$.
\end{proof}

Now consider a sequence $(c_n,d_n)$ in $Q$ that converges to a point $(0,d)$ on $\Delta$, and is such that $c_nd_n=a_nb_n$ for each $n$ (that is, $(c_n,d_n)$ lies on the same leaf as $(a_n,b_n)$). Let $q_n$ be the unique integer such that $(2^{q_n}c_n,d_n/2^{q_n})$ lies in $S$. In other words, $q_n$ is the unique integer such that the quantity $\psi_n=d_n/(c_n2^{2q_n})$ satisfies $1/2 \leq \psi_n <2$.
Observe that
\begin{equation}\label{D}
f^n(c_n,d_n) = \left(2^nc_n\beta(\psi_n),\tfrac{d_n}{2^n\beta(\psi_n)}\right).
\end{equation}
Now,
\[
\psi_n = \frac{d_n}{c_n2^{2q_n}}=\frac{d_n^2}{a_nb_n2^{2q_n}}=\frac{d_n^22^{n-2q_n}}{2^na_nb_n}
\]
Since $2^na_n\beta(\theta_n)\to a$, and by Lemma~\ref{L: beta} $\beta(\theta_n)\to\beta(\theta_0)$, we see that
\[
\psi_n 2^{2q_n-n} \rightarrow \frac{d^2\beta(\theta_0)}{ab}
\]
as $n\rightarrow\infty$. Therefore each limit point of $(\psi_n)$ is equal to $d^2\beta(\theta_0)/(ab)$ multiplied by a power of $2$. Thus either the limit set of $(\psi_n)$ is a subset of $\{\psi_0,2\psi_0\}$ for some $\psi_0\in ]1/2,1[$, or the limit set of $(\psi_n)$ is a subset of $\{\psi_0,2\psi_0,4\psi_0\}$, where $\psi_0=1/2$. Given that $\beta(\psi_0)=\beta(2\psi_0)$ we deduce that
\[
2^nc_n\beta(\psi_n) = \frac{2^na_nb_n\beta(\psi_n)}{d_n} \rightarrow \frac{ab\beta(\psi_0)}{d\beta(\theta_0)}
\]
and
\[
\frac{d_n}{2^n\beta(\psi_n)}\rightarrow 0
\]
as $n\rightarrow\infty$. That is, $f^n(c_n,d_n)$ converges to the point $\left(ab\beta(\psi_0)/(d\beta(\theta_0)),0\right)$ on $\Delta'$, as required.

Finally, we prove that the eight point matching property fails for $f$.

Choose positive real numbers $a$, $b$, and $\delta>1$ such that $b/a< 3/4 < \delta^2b/a$, and $|\delta^2b/a-b/a|<1/100$. Now choose positive numbers $c>a$ and $d<b$ such that $d/c$, $\delta^2 d/c$, $d/a$, $\delta^2 d/a$, $b/c$, and $\delta^2 b/c$ each lie between $1/2$ and $b/a$. A graph of $\beta$ is shown in Figure~\ref{F: beta}.

\begin{figure}[ht]
\centering
\includegraphics[scale=1]{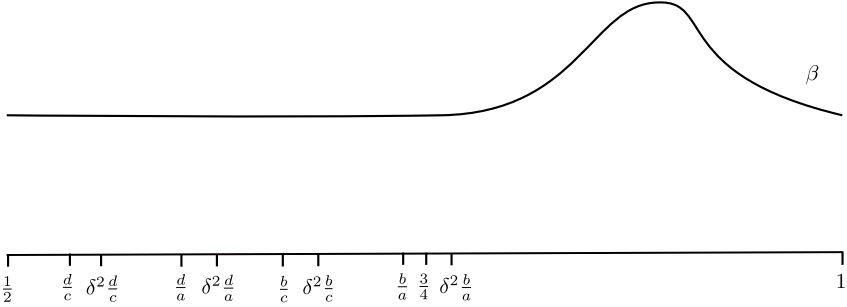}
\caption{The graph of $\beta$.}
\label{F: beta}
\end{figure}

To correspond with the notation for the definition of the eight point matching property in the introduction, let us define $x_1=(0,b)$, $y_1=(0,\delta b)$, $x'_1=(a,0)$, $y'_1=(a/\delta,0)$, $x_2=(0,d)$, $y_2=(0,\delta d)$, $x'_2=(c,0)$, and $y'_2=(c/\delta,0)$. Observe that
\[
\left(\tfrac{c}{2^{2n}},d\right)\to (0,d),\quad f^{2n}\left(\tfrac{c}{2^{2n}},d\right)\to \left(\beta\left(\tfrac{d}{c}\right)c,0\right) = (c,0),
\]
\[
\left(\tfrac{c}{\delta 2^{2n}},\delta d\right)\to \left(0,\delta d\right),\quad f^{2n}\left(\tfrac{c}{\delta 2^{2n}},\delta d\right)\to \left(\beta\left(\delta^2\tfrac{d}{c}\right)\tfrac{c}{\delta},0\right) = \left(\tfrac{c}{\delta},0\right),
\]
and
\[
\left(\tfrac{a}{2^{2n}},d\right)\to (0,d),\quad f^{2n}\left(\tfrac{a}{2^{2n}},d\right)\to \left(\beta\left(\tfrac{d}{a}\right)a,0\right) =(a,0),
\]
\[
\left(\tfrac{a}{\delta 2^{2n}},\delta d\right)\to \left(0,\delta d\right),\quad f^{2n}\left(\tfrac{a}{\delta 2^{2n}},\delta d\right)\to \left(\beta\left(\delta^2\tfrac{d}{a}\right)\tfrac{a}{\delta},0\right) =\left(\tfrac{a}{\delta},0\right),
\]
and
\[
\left(\tfrac{c}{2^{2n}},b\right)\to (0,b),\quad f^{2n}\left(\tfrac{c}{2^{2n}},b\right)\to \left(\beta\left(\tfrac{b}{c}\right)c,0\right) =(c,0),
\]
\[
\left(\tfrac{c}{\delta 2^{2n}},\delta b\right)\to \left(0,\delta b\right),\quad f^{2n}\left(\tfrac{c}{\delta 2^{2n}},\delta b\right)\to \left(\beta\left(\delta^2\tfrac{b}{c}\right)\tfrac{c}{\delta},0\right) =\left(\tfrac{c}{\delta},0\right),
\]
and
\[
\left(\tfrac{a}{2^{2n}},b\right)\to (0,b),\quad f^{2n}\left(\tfrac{a}{2^{2n}},b\right)\to \left(\beta\left(\tfrac{b}{a}\right)a,0\right) = (a,0),
\]
\[
\left(\tfrac{a}{\delta 2^{2n}},\delta b\right)\to \left(0,\delta b\right),\quad f^{2n}\left(\tfrac{a}{\delta 2^{2n}},\delta b\right)\to \left(\beta\left(\delta^2\tfrac{b}{a}\right)\tfrac{a}{\delta},0\right) \neq \left(\tfrac{a}{\delta},0\right).
\]
In terms of the equivalence relation $\sim_f$ on $\Delta\times \Delta'$, these equations show that $(x_2,x_2')\sim_f (y_2,y_2')$, $(x_2,x_1')\sim_f (y_2,y_1')$, and $(x_1,x_2')\sim_f (y_1,y_2')$, but $(x_1,x_1')\nsim_f (y_1,y_1')$. Hence the eight point matching property fails for $f$.

It remains to discuss our second example: an $\folR$-homeomorphism that does not satisfy the four point matching property. Examples of suitable homeomorphisms, without mention of the four point matching property, are found in the literature; see, for instance, \cite[Section 6]{BeLe2003}. In the notation of \cite{BeLe2003}, the $\mathcal{F}_f$-homeomorphism $h_f$, where $\theta(t)=\frac{1}{4} \sin^2(\pi t)$, does not satisfy the four point matching property. To see this, let $(x_k)$ and $(y_k)$, where $y_k\in L(x_k)$, be the convergent sequences in $\mathbb{R}^2$ that correspond to the sequences $(0,1/k)$ and $(1/2,1/k)$ on the chart $P_1$. Then, by switching to the chart $P_2$, it can be checked that $(h_f^k(x_k))$ converges but $(h_f^k(y_k))$ diverges, and hence the four point matching property fails.

\end{document}